\documentclass{amsart}
\RequirePackage{amsmath}
\RequirePackage{amsopn}
\RequirePackage{amsfonts}
\RequirePackage{paralist}
\RequirePackage{amssymb}
\RequirePackage{amsthm}

\usepackage{amsrefs}
\renewcommand\MR[1]{\relax} 
\usepackage{mathrsfs}

\newtheorem{thm}{Theorem}[section]
\numberwithin{equation}{section}

\newtheorem{cor}[thm]{Corollary}
\newtheorem{lemma}[thm]{Lemma}
\newtheorem{prop}[thm]{Proposition}

\theoremstyle{definition}

\theoremstyle{remark}
\newtheorem{remark}[thm]{Remark}

\newtheorem{question}[thm]{Question}

\newtheorem{mycomment}[thm]{Comment}
{\end{mycomment}\endgroup}
\def\mathcs{C^{*}}
\newcommand{\cs}{\ensuremath{\mathcs}}

\DeclareMathSymbol{\rtimes}{\mathbin}{AMSb}{"6F}

\def\T{\mathbf{T}}

\DeclareMathOperator*{\supp}{supp}
\def\set#1{\{\,#1\,\}}
\newcommand\sset[1]{\{#1\}}

\def\restr#1{|_{{#1}}}
\makeatletter
\def\labelenumi{\textnormal{(\@alph\c@enumi)}}
\def\theenumi{\@alph \c@enumi}
\def\labelenumii{\textnormal{(\@roman\c@enumii)}}
\def\theenumii{\@roman \c@enumii}
\newcount\charno
\def\alphapart#1{\charno=96
\advance\charno by#1\char\charno}

\makeatother

\def\<{\langle}
\def\>{\rangle}
\let\ipscriptstyle=\scriptscriptstyle
\def\lipsqueeze{{\mskip -3.0mu}}
\def\ripsqueeze{{\mskip -3.0mu}}
\def\ipcomma{\nobreak\mathrel{,}\nobreak}
\newbox\ipstrutbox
\setbox\ipstrutbox=\hbox{\vrule height8.5pt
width 0pt}
\def\ipstrut{\copy\ipstrutbox}
\def\lip#1<#2,#3>{\mathopen{\relax_{\ipstrut\ipscriptstyle{
#1}}\lipsqueeze
\langle} #2\ipcomma #3 \rangle}
\def\blip#1<#2,#3>{\mathopen{\relax_{\ipstrut
\ipscriptstyle{ #1}}\lipsqueeze\bigl\langle} #2\ipcomma #3 \bigr\rangle}
\def\rip#1<#2,#3>{\langle #2\ipcomma #3
\rangle_{\ripsqueeze\ipstrut\ipscriptstyle{#1}}}
\def\brip#1<#2,#3>{\bigl\langle #2\ipcomma #3
\bigr\rangle_{\ripsqueeze\ipstrut\ipscriptstyle{#1}}}
\def\angsqueeze{\mskip -6mu}
\def\smangsqueeze{\mskip -3.7mu}
\def\trip#1<#2,#3>{\langle\smangsqueeze\langle #2\ipcomma #3
\rangle\smangsqueeze\rangle_{\ripsqueeze\ipstrut\ipscriptstyle{#1}}}
\def\btrip#1<#2,#3>{\bigl\langle\angsqueeze\bigl\langle #2\ipcomma
#3
\bigr\rangle
\angsqueeze\bigr\rangle_{\ripsqueeze\ipstrut\ipscriptstyle{#1}}}
\def\tlip#1<#2,#3>{\mathopen{\relax_{\ipstrut\ipscriptstyle{
#1}}\lipsqueeze \langle\smangsqueeze\langle} #2\ipcomma #3
\rangle\smangsqueeze\rangle}
\def\btlip#1<#2,#3>{\mathopen{\relax_{\ipstrut\ipscriptstyle{
#1}}\lipsqueeze
\bigl\langle\angsqueeze\bigl\langle} #2\ipcomma #3
\bigr\rangle\angsqueeze\bigr\rangle}

\def\ip(#1|#2){(#1\mid #2)}
\def\bip(#1|#2){\bigl(#1 \mid #2\bigr)}
\def\Bip(#1|#2){\Bigl( #1 \bigm| #2 \Bigr)}
\expandafter\ifx\csname BibSpec\endcsname\relax\else
\BibSpec{collection.article}{%
    +{}  {\PrintAuthors}                {author}
    +{,} { \textit}                     {title}
    +{.} { }                            {part}
    +{:} { \textit}                     {subtitle}
    +{,} { \PrintContributions}         {contribution}
    +{,} { \PrintConference}            {conference}
    +{}  {\PrintBook}                   {book}
    +{,} { }                            {booktitle}
    +{,} { }                            {series}
    +{,} { \voltext}                    {volume}
    +{,} { }                            {publisher}
    +{,} { }                            {organization}
    +{,} { }                            {address}
    +{,} { \PrintDateB}                 {date}
    +{,} { pp.~}                        {pages}
    +{,} { }                            {status}
    +{,} { \PrintDOI}                   {doi}
    +{,} { available at \eprint}        {eprint}
    +{}  { \parenthesize}               {language}
    +{}  { \PrintTranslation}           {translation}
    +{;} { \PrintReprint}               {reprint}
    +{.} { }                            {note}
    +{.} {}                             {transition}
}
\BibSpec{article}{%
    +{}  {\PrintAuthors}                {author}
    +{,} { \textit}                     {title}
    +{.} { }                            {part}
    +{:} { \textit}                     {subtitle}
    +{,} { \PrintContributions}         {contribution}
    +{.} { \PrintPartials}              {partial}
    +{,} { }                            {journal}
    +{}  { \textbf}                     {volume}
    +{}  { \PrintDatePV}                {date}
    +{,} { \eprintpages}                {pages}
    +{,} { }                            {status}
    +{,} { \PrintDOI}                   {doi}
    +{,} { available at \eprint}        {eprint}
    +{}  { \parenthesize}               {language}
    +{}  { \PrintTranslation}           {translation}
    +{;} { \PrintReprint}               {reprint}
    +{.} { }                            {note}
    +{.} {}                             {transition}
}
\BibSpec{book}{%
    +{}  {\PrintPrimary}                {transition}
    +{,} { \textit}                     {title}
    +{.} { }                            {part}
    +{:} { \textit}                     {subtitle}
    +{,} { \PrintEdition}               {edition}
    +{}  { \PrintEditorsB}              {editor}
    +{,} { \PrintTranslatorsC}          {translator}
    +{,} { \PrintContributions}         {contribution}
    +{,} { }                            {series}
    +{,} { \voltext}                    {volume}
    +{,} { }                            {publisher}
    +{,} { }                            {organization}
    +{,} { }                            {address}
    +{,} { pp.~}                        {pages}
    +{,} { \PrintDateB}                 {date}
    +{,} { }                            {status}
    +{}  { \parenthesize}               {language}
    +{}  { \PrintTranslation}           {translation}
    +{;} { \PrintReprint}               {reprint}
    +{.} { }                            {note}
    +{.} {}                             {transition}
}
\fi
\newcommand\go{G^{(0)}} 

\def\g[#1,#2]{{}_{G}[#1,#2]} 
\def\h[#1,#2]{[#1,#2]_{H}}

\let\phi\varphi
\newcommand\zop{Z^{\text{op}}}
\newcount\hours
\newcount\minutes       
\def\timeofday{
\hours=\time
\minutes=\hours
\divide\hours by60
\multiply\hours by60
\advance\minutes by-\hours
\divide\hours by60
\ifnum\hours>9\else0\fi\the\hours:\ifnum\minutes>9\else
0\fi\the\minutes}
\def\predate{\date{\the\day\ \ifcase\month\or
  January\or February\or March\or April\or May\or June\or July\or
        August\or September\or October\or November\or
           December\fi\ \the\year\ --- \timeofday\ --- Preliminary
                  Version}}
     
\usepackage[normalem]{ulem} 
\usepackage{color}
\definecolor{Dgreen}{cmyk}{0.93,0.33,0.92,0.25} 

\begin{document}

\title{\boldmath Haar Systems on Equivalent
Groupoids}
\author{Dana P. Williams} 
\address{Department of Mathematics \\ Dartmouth College \\ Hanover, NH
03755 \\ USA}
\email{dana.williams@dartmouth.edu}

\thanks{This work was supported in part by a grant from the Simons
  Foundation.}  

\date{15 January 2015}

\begin{abstract}
For second countable locally compact Hausdorff groupoids, the
property of possessing a Haar system is preserved by equivalence.
\end{abstract}
\maketitle

\section{Introduction}
\label{sec:introduction}

Beginning with the publication of Renault's seminal paper
\cite{ren:groupoid}, locally compact groupoids have been an especially
important way to construct operator algebras.  Just as with the time
honored group \cs-algebra construction, this is done by turning
$C_{c}(G)$ into a convolution algebra and then completing.  In the group
case, there is always a (left) Haar measure on $G$ which allows us to
form the convolution product.  In the groupoid case, the natural
convolution formula requires a family of measures $\lambda^{u}$ with
support $G^{u}=\set{x\in G:r(x)=u}$ for each $u\in\go$.  We want the
family to be left-invariant in that $x\cdot \lambda^{s(x)} =
\lambda^{r(x)}$ where $x\cdot
\lambda^{s(x)}(E)=\lambda^{s(x)}(x^{-1}E)$.  In order that the
convolution formula return a continuous function, we need the
continuity condition that
\begin{equation*}
  u\mapsto \int_{G}f(x)\,d\lambda^{u}(x)
\end{equation*}
be continuous for all $f\in C_{c}(G)$ (the necessity is the main
result in \cite{sed:pams86}).  Such a family
$\sset{\lambda^{u}}_{u\in\go}$ is called a (continuous) \emph{Haar
  system} for $G$.  An annoying gap in the theory is that there is no
theorem guaranteeing Haar systems exist.  The only significant
positive existence result I am aware of is that if $\go$ is open in
$G$ and the range map is open (and hence the source map as well), then
the family consisting of counting measures is always a Haar
system. Groupoids with $\go$ open and for which the range map is open
are called \emph{\'etale}.  (It is also true that Lie Groupoids
necessarily have Haar systems \cite{pat:groupoids99}*{Theorem~2.3.1},
but this result is crucially dependent on the manifold structure and
hence not in the spirit of this note.)  It is also well known that if
$G$ is any locally compact groupoid with a Haar system, then its range
and source maps must be open.  (This is a consequence of
Remark~\ref{rem-haar-r-sys} and Lemma~\ref{lem-pi-sys-open}.)  Thus if
a locally compact groupoid has a range map which is not open, then it
can't possess a Haar system.  Such groupoids do exist; for example,
see \cite{sed:pams86}*{\S3}.  However, to the best of my knowledge,
there is no example of a locally compact groupoid with open range and
source maps which does not possess a Haar system. I have yet to find
an expert willing to conjecture (even off the record) that all such
groupoids need have Haar systems, but the question remains open.

The purpose of this note is to provide some additional examples where
Haar systems must exist.  The main result being that if $G$ and $H$
are equivalent second countable locally compact groupoids (as defined
in \cite{mrw:jot87} for example), and if $G$
has a Haar system, then so does $H$.  Since equivalence is such a
powerful tool, this result gives the existence of Haar systems on a
great number of interesting groupoids.  For example, every transitive
groupoid with open range and source maps has a Haar system
(Proposition~\ref{prop-trans}).  

The proof given here depends on several significant results from the
literature.  The first is that if $\pi:Y\to X$ is a continuous, open
surjection with $Y$ second countable, then there is a family of Radon
measures $\sset{\beta^{x}}$ on $Y$ such that $\supp
\beta^{x}=\pi^{-1}(x)$ and
\begin{equation*}
  x\mapsto \int_{Y} f(y)\,d\beta^{x}(y)
\end{equation*}
is continuous for all $f\in C_{c}(Y)$.  (This result is due to
Blanchard who makes use of a Theorem of Michael's \cite{mic:am56}.) 
The second is the
characterization in \cite{kmrw:ajm98}*{Proposition~5.2} of when the
imprimitivity groupoid of free and proper $G$-space has a Haar
system.  The third is the concept of a Bruhat section or cut-off
function.  These are used in
\cite{bou:integrationII04}*{Chapter~7} to construct invariant
measures.  They also appear prominently in
 \cite{ren:xx14}*{Lemma~25} and \cite{tu:doc04}*{\S6}.

Since Blanchard's result requires separability, we can only consider
second countable groupoids here.

I would like to thank Marius Ionescu, Paul Muhly, Erik van Erp, Aidan
Sims, and especially Jean Renault for helpful comments and
discussions.

\section{The Theorem}
\label{sec:main-theorem}

\begin{thm}
  \label{thm-main}
  Suppose that $G$ is a second countable, locally compact Hausdorff
  groupoid with a Haar system $\sset{\lambda^{u}}_{u\in\go}$.  If $H$ 
  is a second countable, locally compact groupoid which is equivalent
  to $G$, then $H$ has a Haar system.
\end{thm}

As in \cite{ren:jot87}*{p.~69} or \cite{anaren:amenable00}*{Definition~1.1.1},
if $\pi:Y\to X$ is a continuous map between locally compact spaces $Y$
and $X$, then a \emph{$\pi$-system} is a family of (positive) Radon
measures $\set{\beta^{x}:x\in X}$ on $Y$ such that $\supp \beta^{x}\subset \pi^{-1}(x)$ and for every $f\in C_{c}(Y)$, the function
  \begin{equation*}
    x\mapsto \int_{Y}f(y)\,d\beta^{x}(y)
  \end{equation*}
is continuous.
We say that $\beta$ is \emph{full} if $\supp \beta^{x}=\pi^{-1}(x)$
for all $x\in X$.  

If $Y$ and $X$ are both (left) $G$-spaces and $\pi$ is
equivariant, then we say $\beta$ is \emph{equivariant} if $\gamma\cdot
\beta^{x} = \beta^{\gamma\cdot x}$ where $\gamma\cdot
\beta^{x}(E)=\beta^{x}(\gamma^{-1}\cdot E)$ for all $(\gamma,x)\in
G*X=\set{(\gamma,x):s(\gamma)=r(x)}$.  Alternatively,
\begin{equation*}
  \int_{Y}f(\gamma\cdot y)\,d\beta^{x}(y)=\int_{Y}
  f(y)\,d\beta^{\gamma\cdot x}(y)
\end{equation*}
for all $f\in C_{c}(Y)$ and  $(\gamma,x)\in G*X$.

\begin{remark}
  \label{rem-haar-r-sys}
  It is useful to keep in mind that a
 Haar system on $G$ is a full, equivariant
$r$-system on $G$ for the range map $r:G\to\go$.
\end{remark}

In many cases, such as \cite{kmrw:ajm98}*{\S5}, $\pi$-systems are
reserved for continuous \emph{and open} maps $\pi:Y\to X$.  In the
case of full systems, the next lemma implies that there is no loss in
generality. (This part of the result does not require second
countability.)  Conversely, if $Y$ is second countable $\pi$ is an
open surjection, then Blanchard has proved that full $\pi$ systems
must exist.  Blanchard's result will be crucial in the proof of the
main result.

\begin{lemma}[Blanchard]
  \label{lem-pi-sys-open}
  Suppose that $\pi:Y\to X$ is a continuous surjection between second
  countable locally compact Hausdorff spaces.  Then $\pi$ is
  open if and only if it admits a full $\pi$ system.
\end{lemma}
\begin{proof}
  Suppose that  $\beta$ is a full $\pi$ system.  To show that $\pi$ is
  open, we appeal to the usual lifting argument as in
  \cite{wil:crossed}*{Proposition~1.15}.  Thus we assume that
  $x_{i}\to \pi(y)$ is a convergent net.  It will suffice to produce a
  subnet $\sset{x_{j}}_{j\in J}$ and elements $y_{j}\in \pi^{-1}(x_{j})$ such
  that $y_{j}\to y$.

To this end, let
\begin{equation*}
  J=\set{(i,V):\text{$V$ is an open neighborhood of $y$ and
      $\pi^{-1}(x_{i})\cap V\not=\emptyset$}}.
\end{equation*}
We need to see that $J$ is directed in the expected way: $(i,V)\ge
(j,U)$ if $i\ge j$ and $V\subset U$.  So let $(k,V)$ and $(j, U)$ be
in $J$.  Let $f\in C_{c}^{+}(G)$ be such that $f(y)=1$ and $\supp
f\subset V\cap U$.  Then
\begin{equation*}
  \beta(f)(x_{i})=\int_{Y} f(y)\,d\beta^{x_{i}}(y) \to \int_{Y} f(y)
  \, \beta^{\pi(y)}(y)=\beta(f)(\pi(y)).
\end{equation*}
The latter is nonzero since $\beta^{\pi(y)}$ has full support.  Hence
there is a $k\ge i$ and $k\ge j$ such that
\begin{equation*}
  \int_{Y}f(y)\,d\beta^{x_{i}}(y)\not=0.
\end{equation*}
It follows that $(i,U\cap V)\in J$.  

Thus if we let $x_{(i,V)}=x_{i}$, then $\sset{x_{(i,V)}}_{(i,V)\in J}$
is a subnet.  If we let $y_{(i,V)}$ be any element of
$\pi^{-1}(x_{i})\cap V$, then $y_{(i,V)}\to y$.  This suffices.

  The converse is much more subtle, and is due to Blanchard
  \cite{bla:bsmf96}*{Proposition~3.9}. 
\end{proof}

We also need what is sometimes called a \emph{Bruhat section} or
\emph{cut-off} function for $\pi$.  The construction is modeled after
Lemma~1 in Appendix~I for \cite{bou:integrationII04}*{Chapter 7}.
Recall that a subset $A\subset Y$ is called $\pi$-compact if $A\cap
\pi^{-1}(K)$ is compact whenever $K$ is compact in $X$.  We write
$C_{c,\pi}(Y)$ for the set of continuous functions on $Y$ with
$\pi$-compact support.

\begin{lemma}
  \label{lem-phi-exists}
  Let $\pi:Y\to X$ be a continuous open surjection between second
  countable locally compact Hausdorff spaces.  Then there is a
  $\phi\in C_{c,\pi}^{+}(Y)$ such that $\pi\bigl(\set{y\in Y:\phi(y)>0}
\bigr) = X$.
\end{lemma}
\begin{proof}
  Let $\mathscr V=\set{V_{i}}$ be a countable, locally finite cover of
  $X$ by pre-compact open sets $V_{i}$.  Let $\sset{\alpha_{i}}$ be a
  partition of unity on $X$ subordinate to $\mathscr V$.  Let
  $\phi_{i}\in C_{c}^{+}(Y)$ be such that $\pi\bigl(\set{y\in
    Y:\phi_{i}>0}\bigr) \supset V_{i}$.  Then we can define
  \begin{equation*}
    \phi(y) =\sum_{i} \phi_{i}(y) \alpha_{i}\bigl(\pi(y)\bigr).
  \end{equation*}
  Since $\mathscr V$ is locally finite, the above sum is finite in a
  neighborhood of any $y\in Y$.  Hence $\phi$ is well-defined and
  continuous. Local finiteness also implies that every compact subset
  of $X$ meets at most finitely many $V_{i}$.  Since
  $\sset{\alpha_{i}}$ is subordinate to $\mathscr V$, it follows that
  $\phi$ has $\pi$-compact support.  If $x\in X$, then there is an $i$
  such that $\alpha_{i}(x)>0$.  Then there is a $y$ such that
  $\phi_{i}(y)>0$ and $\pi(y)=x$.  Hence the result.
\end{proof}

\begin{proof}[Proof of Theorem~\ref{thm-main}]
  Let $Z$ be a $(G,H)$-equivalence.  
  Then the opposite module, $\zop$, is a $(H,G)$ equivalence.
  Therefore, in view of \cite{kmrw:ajm98}*{Proposition~5.2}, it will
  suffice to produce a full $G$-equivariant $s_{\zop}$-system for the
  structure map $s_{\zop}:\zop\to\go$.  Equivalently, we need a full
  equivariant $r_{Z}$-system for the map $r_{Z}:Z\to\go$.\footnote{The
    gymnastics with the opposite space is just to accommodate a
    preference for left-actions.  This has the advantage of making
    closer contact with the literature on $\pi$-systems.}  Hence the
  main Theorem is a consequence of
  Proposition~\ref{prop-jean-proper-map} below.
\end{proof}

The following proposition is even more that what is called for in the
proof of Theorem~\ref{thm-main}: it shows that every proper $G$-space
has a full equivariant $r$-system for the moment map whether the
action is free or not.  It should be noted that pairs $(X,\alpha)$
where $X$ is a proper $G$-space and $\alpha$ and equivariant
$r$-system play an important role in the constructions in
\cite{ren:jot87} and \cite{holren:xx14}.  This makes the assertion
that such $\alpha$'s always exist even more interesting.
\begin{prop}
  \label{prop-jean-proper-map}
  Let $G$ be a locally compact Hausdorff groupoid with a Haar system
  $\sset{\lambda^{u}}_{u\in\go}$.  Suppose that $Z$ is a proper
  $G$-space.  Then there is a full
  equivariant $r_{Z}$-system $\sset{\nu^{u}}_{u\in\go}$ for the moment
  map $r_{Z}:Z\to\go$. 
\end{prop}

\begin{proof}
  Blanchard's Lemma~\ref{lem-pi-sys-open} implies that there is a full
  $r_{Z}$-system $\beta=\sset{\beta^{u}}_{u\in\go}$.  
  The idea of the proof is to use the Haar system on $G$ to
  average this system to create and
  equivariant system.  The technicalities are provided by the next
  lemma. Notice that since $G$ acts
  properly, the orbit map $q:Z\to G\backslash Z$ is a continuous open
  surjection between locally compact Hausdorff spaces.

  \begin{lemma}
    \label{lem-phi-fcns}
    Let $G$, $Z$, $\lambda$ and $\beta$ be as above.
    \begin{enumerate}
    \item If $F\in C_{c}(G\times Z)$, then
      \begin{equation*}
        \Phi(F)(g,u)=\int_{Z} F(g,z)\,d\beta^{u}(z)
      \end{equation*}
defines an element of $C_{c}(G\times\go)$.
\item If $f\in C_{c}(Z)$ and $\phi\in C_{c,q}(Z)$, then
  \begin{equation*}
    \Psi_{\phi}(f)(g)=\int_{Z} f(g\cdot z)\phi(z) \,d\beta^{s(g)}(z)
  \end{equation*}
defines an element of $C_{c}(G)$.
    \end{enumerate}
  \end{lemma}
  \begin{proof}
    (a) This is straightforward if $F(g,z)=f(g)\phi(z)$ with $f\in
    C_{c}(G)$ and $\phi\in C_{c}(Z)$.  But we can approximate $F$ in
    the inductive limit topology with sums of such functions.

    (b) Let $L=\supp \phi\cap q^{-1}\bigl(q(\supp f)\bigr)$. By assumption on
    $\phi$, $L$ is compact.  Since $G$ acts properly on $Z$, the set
    \begin{equation*}
      P(\supp f,L)=\set{g\in G:g\cdot L\cap \supp f\not=\emptyset}
    \end{equation*}
is compact.  It follows that
    \begin{equation*}
      F(g,z)=f(g\cdot z)\phi(z)
    \end{equation*}
defines an element of $C_{c}(G\times Z)$.  Then
\begin{equation*}
  \Psi_{\phi}(f)(g)=\Phi(F)(g,s(g)).
\end{equation*}
The assertion follows.
  \end{proof}

  Using Lemma~\ref{lem-phi-exists}, we fix $\phi\in
  C_{c,q}^{+}(Z)$ such that $q\bigl(\set{z:\phi(z)>0}
  \bigr)=G\backslash Z$. Then we define a Radon measure on $C_{c}(Z)$ by
\begin{equation}
  \label{eq:1}
\nu^{u}(f)=\int_{G}\int_{Z} f(g\cdot z)\phi(z)\,d\beta^{s(g)}(z)
\,d\lambda^{u}(g) =\int_{G}\Psi_{\phi}(f)(g) \,d\lambda^{u}(g).
\end{equation}

Since $\lambda$ is a Haar system and $\Psi_{\phi}(f)\in C_{c}(G)$, we
see immediately that
\begin{equation*}
  u\mapsto \nu^{u}(f)
\end{equation*}
is continuous.

Clearly, $\supp \nu^{u}\subset r_{Z}^{-1}(u)$.  Suppose $r_{Z}(w)=u$
and $f\in C_{c}^{+}(Z)$ is such that $f(w)>0$.
Then there is a $z'\in\set{z:\phi(z)>0}$ such that
$q(z')=q(w)$.  Hence there is a $g\in
G$ such that $g\cdot z'=w$.  Note that $r_{Z}(z')=s(g)$ and
$r(g)=r_{Z}(g\cdot z)=r_{Z}(w)=u$.  Since $\beta^{s(g)}$ has full
support and since everything in sight is continuous and non-negative,
\begin{equation*}
  \Psi_{\phi}(g) =\int_{Z} f(g\cdot z)\phi(z)\,\beta^{s(g)}(z)>0.
\end{equation*}
Hence $\nu^{u}(f)>0$ and we have
\begin{equation*}
  \supp \nu^{u} = r_{Z}^{-1}(u).
\end{equation*}
 
Hence to complete the proof of the theorem, we just need to establish
equivariance.  But
\begin{align*}
  \int_{Z} f(g'\cdot z)\,d\nu^{s(g')}(z) &= \int_{G}\int_{Z}
  f(g'g\cdot z)\phi(z)\,d\beta^{s(g)}(z)\,d\lambda^{s(g')}(g) \\
&= \int_{G}\Psi_{\phi}(f)(g'g)\,d\lambda^{s(g')}(g) \\
\intertext{which, since $\lambda$ is a Haar system on $G$, is}
&= \int_{G}\Psi_{\phi}(g) \,d\lambda^{r(g')}(g) \\
&=\int_{G}\int_{Z} f(g\cdot z)\phi(z)\,d\beta^{s(g)}(z)
\,d\lambda^{r(g')}(g) \\
\intertext{which, since $g'\cdot s(g')=r(g')$, is}
&=\int_{Z}f(z)\,d\nu^{g'\cdot s(g')}(z).
\end{align*}
This completes the proof.
\end{proof}

Proposition~\ref{prop-jean-proper-map} is interesting even for a group
action.  The result itself is no doubt known to experts, but is
amusing none-the-less.
\begin{cor}
  \label{cor-jean-group}
  Suppose that $G$ is a locally compact group acting properly on a
  space $X$.  Then $X$ has at least one invariant measure with full
  support. 
\end{cor}

\section{Examples and Comments}
\label{sec:exampl-comm-future}

As pointed out in the introduction, any \'etale groupoid $G$ has a
Haar system.  As a consequence of Theorem~\ref{thm-main}, any second
countable groupoid equivalent to $G$ has a Haar system (provided $G$
is second countable).  In this section, I want to look at some
additional examples.  In some cases it is possible and enlightening to
describe the Haar system in finer detail.

\subsection{Proper Principal Groupoids}
\label{sec:prop-princ-group}

Recall that $G$ is called principal if the natural action of $G$ on
$\go$ given by $x\cdot s(x)=r(x)$ is free. We call $G$ proper if this
action is proper in that $(x,s(x))\mapsto (r(x),s(x))$ is proper from
$G\times \go \to\go\times\go$.  If $G$ is a proper principal groupoid
with open range and source maps, 
then the orbit space $G\backslash \go$ is locally compact Hausdorff,
and it is straightforward to check that
$\go$ implements an equivalence between $G$ and the orbit space
$G\backslash \go$. Since the orbit space clearly has a Haar system,
the following is a simple corollary of
Theorem~\ref{thm-main}.

\begin{prop}[Blanchard]
  \label{prop-proper-prin}
  Every second countable proper principle groupoid with open range and
  source maps has a Haar system.
\end{prop}
\begin{remark}
  \label{rem-blanchard}
  If $G$ is a second countable proper principle groupoid with open
  range and source maps, then the orbit map $q:\go\to G\backslash \go$
  sending $u$ to $\dot u$ is continuous and open.  Hence Blanchard's
  Lemma~\ref{lem-pi-sys-open} implies there is a full $q$-system
  $\sset{\beta^{\dot u}}_{\dot u\in G\backslash \go}$.  It is not hard
  to check that $\lambda^{u}=\delta_{u}\times \beta^{\dot u}$ is a
  Haar system for
  \begin{equation*}
    G_{q}:=\set{(u,v)\in\go\times\go:\dot u = \dot v}.
  \end{equation*}
Since $x\mapsto (r(x),s(x))$ is a groupoid isomorphism of $G$ and
$G_{q}$, we get an elementary description for a Haar system on $G$.
\end{remark}

While there certainly exist groupoids that fail to have open range and
source maps --- and hence cannot have Haar systems --- most of these
examples are far from proper and principal.  In fact the examples I've seen are
all group bundles which are as a far from principal as possible.  This
poses an interesting question.

\begin{question}
  Must a second countable, locally compact, proper principal groupoid have
  open range and source maps?
\end{question}

\subsection{Transitive Groupoids}
\label{sec:transitive-groupoids}

Recall that a groupoid is called \emph{transitive} if the natural
action of $G$ on $\go$ given by $x\cdot s(x):=r(x)$ is transitive.  If
$G$ is transitive and has open range and source maps, then $G$ is
equivalent to any of its stability groups $H=G_{v}^{v}=\set{x\in
  G:r(x)=u=s(x)}$ for $v\in\go$; the equivalence is given by $G_{v}$
with the obvious left $G$-action and right $H$ action.  Second
countability is required to see that the restriction of the range map
to $G_{v}$ onto $\go$ is open.\footnote{Proving the openness of $r\restr{G_{v}}$ is nontrivial.  It
  follows from \cite{ram:jfa90}*{Theorem~2.1} or Theorems 2.2A~and
  2.2B in \cite{mrw:jot87}.   The assertion and
  equivalence fail without the 
second countability assumption as observed in
\cite{mrw:jot87}*{Example~2.2}.

It should be noted that in both
  \cite{ram:jfa90} and \cite{mrw:jot87} openness of the range and source maps on
  a topological groupoid is a standing assumption.}  Since locally
compact groups always have a Haar measure, the following is an
immediate consequence of Theorem~\ref{thm-main}.  (Similar assertions
can be found in \cite{sed:pria76}.)
\begin{prop}[Seda]
  \label{prop-trans}
  If $G$ is a second countable, locally compact transitive groupoid
  with open range and source maps,
  then $G$ has a Haar system.
\end{prop}


As before, I don't know the answer to the following.

\begin{question}
  Must a second countable, locally compact, transitive groupoid have
  open range and source maps?
\end{question}

\subsection{Blowing Up the Unit Space}
\label{sec:blowing-up-unit}

While there are myriad ways groupoid
equivalences arise in applications, one standard technique
deserves special mention (see 
\cite{txlg:acens12} for example).  Suppose that $G$ is a second
countable locally compact groupoid with a Haar system (or at least
open range and source maps).  Let $f:Z\to \go$ be a continuous and
open map.  Then we can form the groupoid
\begin{equation}
  \label{eq:2}
  G[Z]=\set{(z,g,w)\in Z\times G\times Z: \text{$f(z)=r(g)$ and
      $s(g)=f(w)$}} .
\end{equation}
(The operations are as expected: $(z',g',z)(z,g,w)=(z',g'g,w)$ and
$(z,g,w)^{-1} = (w,g^{-1},z)$.)  The idea being that we use $f$ to
``blow-up'' the unit space of $G$ to all of $Z$.  If $\phi:G[Z]\to G$
is the homormorphism $(z,g,w)\mapsto g$, then we get a
$(G[Z],G)$-equivalence given by ``the graph of $\phi$'' (see
\cite{kmrw:ajm98}*{\S6}): 
\begin{equation*}
W = \set{(z,g)\in Z\times G:f(z)=r(g)}.
\end{equation*}
The left $G[Z]$-action is given by $(z,g,w)\cdot (w,g')=(z,gg')$ and
the right $G$-action by $(w,g')\cdot g = (w,g'g)$.  The openness of
the range map for $G$ are required to see that the structure map
$r_{W}:W\to Z$ is open, while the openness of $f$ is required to see
that $s_{W}:W\to \go$ is open.  Assuming $G$ has a Haar system and $Z$
is second countable, Theorem~\ref{thm-main} implies $G[Z]$ has a Haar
system.  However in this case we can do a bit better and write down a
tidy formula for the Haar system on the blow-up.  We still require
Blanchard's Lemma~\ref{lem-pi-sys-open} that there is a full
$f$-system for any continuous open map $f:Z\to \go$ (provided that $Z$
is second countable).

\begin{prop}
  \label{prop-blow-up-haar}
  Suppose that $G$ is a locally compact Hausdorff groupoid with a Haar
  system $\sset{\lambda^{u}}_{u\in\go}$, and that $Z$ is second
  countable.  Let $f:Z\to\go$ be a continuous open map, and let
  \begin{equation*}
    G[Z]=\set{(w,g,z)\in Z\times G\times Z: \text{$f(w)=r(g)$ and $s(g)=f(z)$}}
  \end{equation*}
be the ``blow-up'' of $G$ by $f$.  If $\sset{\beta^{u}}_{u\in\go}$ is
a full $f$-system, then we get a Haar system
$\sset{\kappa^{z}}_{z\in Z}$ on $G[Z]$ given by
\begin{equation*}
  \kappa^{z}(f)=\int_{G}\int_{Z}
  f(z,g,w)\,d\beta^{s(g)}(w)\,d\lambda^{f(z)}(g). 
\end{equation*}
\end{prop}

The proof is relatively straightforward.

\subsection{Imprimitivity Groupoids}
\label{sec:impr-group}

If $X$ is a free and proper right $G$-space, then assuming $G$ has
open range and source maps, we can form the imprimitivity groupoid
$G^{Z}$ as
in \cite{muhwil:plms395}*{pp.~119+}.  Specifically we let $G^{Z}$ be
the quotient of $X*_{s}X=\set{(x,y)\in X\times X:s(x)=s(y)}$ by the
diagonal right $G$-action.  Then $G^{Z}$ is a groupoid with respect to
the operations $[x,y][y,z]=[x,z]$ and $[x,y]^{-1}=[y,z]$.  Furthermore
$X$ implements an equivalence between $G^{Z}$ and $G$.  Again we can
apply Theorem~\ref{thm-main}.
\begin{prop}
  \label{prop-imprimitivity}
  Suppose that $G$ is a second countable, locally compact Hausdorff groupoid
  with a Haar system.  Let $X$ be a free and proper right $G$-space.
  Then the imprimitivity groupoid $G^{Z}$ has a Haar system.
\end{prop}

\begin{remark}
  \label{rem-kmrw}
  In \cite{kmrw:ajm98}*{\S\S 9-10}, we associated a group
  $\operatorname{Ext}(G,\T)$ to any second countable, locally compact
  Hausdorff groupoid $G$.  In \cite{kmrw:ajm98}*{Theorem~10.1} we
  showed that $\operatorname{Ext}(G,\T)$ was naturally isomorphic to
  the Brauer group $\operatorname{Br}(G)$.  The definition of
  $\operatorname{Ext}(G,\T)$ required we consider the space
  $\mathcal{P}(G)$ of all free and proper right $G$-spaces $X$
  \emph{such that $G^{X}$ has a Haar system}.  In view of
  Proposition~\ref{prop-imprimitivity}, $\mathcal{P}(G)$ becomes simply
  the collection of all free and proper right $G$-spaces.
\end{remark}

\bibliographystyle{amsxport}
%


\def\noopsort#1{}\def\cprime{$'$} \def\sp{^}

\end{document}